\documentclass[12pt,leqno,amscd, amsfonts, amssymb,pstricks,verbatim]{amsart}
\usepackage{amsfonts}
\usepackage{amscd}
\usepackage{mathrsfs}
\usepackage{amsmath}
\usepackage{amssymb}
\usepackage{hyperref}
\usepackage{bookmark}
\usepackage{bm}

\def\a{\alpha}
\def\r{\gamma}
\def\b{\beta}

\def\Z{\mathbb{Z}}
\def\N{\mathbb{N}}

\def\C{\mathbb{C}}

\numberwithin{equation}{section}
\newtheorem{theo}{Theorem}[section]
\newtheorem{defi}[theo]{Definition}
\newtheorem{coro}[theo]{Corollary}
\newtheorem{lemm}[theo]{Lemma}
\newtheorem{prop}[theo]{Proposition}

\newtheorem{case}{Case}

\allowdisplaybreaks

\begin{document}

\title[Tensor product modules over the affine-Virasoro algebra of type $A_1$]{A new class of irreducible modules over the affine-Virasoro algebra of type $A_1$}

\author{Qiu-Fan Chen, Yu-Feng Yao}

\address{Department of Mathematics, Shanghai Maritime University,
 Shanghai, 201306, China.}\email{chenqf@shmtu.edu.cn}
\address{Department of Mathematics, Shanghai Maritime University,
 Shanghai, 201306, China.}\email{yfyao@shmtu.edu.cn}



\thanks{This work is supported by National Natural Science Foundation of China (Grant Nos. 11801363, 11771279 and 12071136).}

\begin{abstract}
In this paper, we construct a class of non-weight modules over the affine-Virasoro algebra of type $A_1$ by taking tensor products of a finite number of irreducible modules  $M(\lambda, \alpha, \beta, \gamma)$ with irreducible highest weight modules $V(\eta, \epsilon, \theta)$. We obtain the necessary and sufficient conditions for such tensor product modules to be irreducible, and determine the necessary and sufficient conditions for such two modules to be isomorphic. We also compare these modules with other known non-weight modules, showing that  these irreducible modules are new.
\end{abstract}

\maketitle

\setcounter{tocdepth}{1}\tableofcontents
\begin{center}
\end{center}

\section{Introduction}
Throughout the paper, we denote by $\C ,\,\Z,\,\C^*,\,\Z_+,\,\N$ the sets of complex numbers, integers, nonzero complex numbers, nonnegative integers and positive integers, respectively. All algebras (modules, vector spaces) are assumed to be  over $\C$. For a Lie algebra $\mathfrak{g}$, we use $U(\mathfrak{g})$ to denote the universal enveloping algebra of $\mathfrak{g}$.  More generally, for a subset $X$ of $\mathfrak{g}$, we use $U(X)$ to denote the universal enveloping algebra of the subalgebra of $\mathfrak{g}$ generated by $X$.

It is well known that representation theory of the Virasoro algebra and affine Kac-Moody Lie algebras plays an important role both in physics and in mathematics. The Virasoro algebra acts on any (except when the level is negative the dual coxeter number) highest weight module of the affine Lie algebra through the use of the famous Sugawara operators.  The affine Lie algebras admit representations on the Fock space and hence admit representations of the Virasoro algebra. This close
relationship strongly suggests that they should be considered simultaneously, i.e., as one algebraic structure, and hence has led to the definition of the so-called affine-Virasoro algebra. It is known that in literature this algebra is also named the conformal current algebra \cite{CK,K}, the entire gauge algebra \cite{EFK}. The physical context in which the affine-Virasoro algebra appears
is a two-dimensional conformal field theory on the circle with an internal symmetry algebra. In particular, the even part of  the $N=3$ superconformal algebra \cite{CL} is just the affine-Virasoro algebra of type $A_1$. The affine-Virasoro algebra of type $A_1$, denoted by $\mathfrak L$,  is defined as the Lie algebra with $\C$-basis $\{e_i,\, f_i,\,h_i,\, d_i,\,C\mid i\in\Z\}$ subject to the following Lie brackets:
\begin{equation*}\label{L-action}
\aligned
&[e_i,f_j]=h_{i+j}+i\delta_{i+j,0}C,\\
&[h_i,e_j]=2e_{i+j},\quad [h_i,f_j]=-2f_{i+j},\\
&[d_i,d_j]=(j-i)d_{i+j}+\delta_{i+j,0}\frac{i^3-i}{12}C,\\
&[d_i,h_j]=jh_{i+j},\quad [h_i,h_j]=-2i\delta_{i+j,0}C,\\
&[d_i,e_j]=je_{i+j},\quad [d_i,f_j]=jf_{i+j},\\
&[e_i,e_j]=[f_i,f_j]=[C,\mathfrak{L}]=0.
\endaligned
\end{equation*}
It is clear that  $\mathfrak{h}:=\C d_0+\C h_0+\C C$ is the Cartan subalgebra  of $\mathfrak{L}$. Moreover, $\mathfrak L$ admits a triangular decomposition:
$$\mathfrak{L}=\mathfrak{L}_{-}\oplus\mathfrak{h}\oplus\mathfrak{L}_{+},$$
where
\begin{eqnarray*}
\mathfrak{L}_{-}={\rm span\,}_{\C}\{e_{-i}, f_{-i},  h_{-i}, d_{-i}, f_0 \mid i\in\N\}
\end{eqnarray*}
and \begin{eqnarray*}\mathfrak{L}_{+}={\rm span\,}_{\C}\{e_{i}, f_{i},  h_{i}, d_{i}, e_0\mid i\in\N\}.
\end{eqnarray*}

 Highest weight representations and integrable  representations of the affine-Virasoro algebras have been  extensively studied  (cf.~\cite{B}, \cite{EJ}, \cite{GHL}, \cite{JY}-\cite{LQ},  \cite{XH}). Quite recently, the authors  gave the classification of irreducible quasi-finite modules  over the affine-Virasoro algebras in \cite{LPX}.

In recent years, many authors constructed various irreducible non-Harish-Chandra modules and irreducible non-weight modules (cf.~\cite{BM}, \cite{CTZ, CZ}, \cite{HCS}, \cite{MW}-\cite{TZ}). In particular, J. Nilsson \cite{N} constructed  a class of  $\mathfrak{sl_{n+1}}$-modules that are free of rank one when restricted to the Cartan subalgebra.  Since then, this kind of non-weight modules, which many authors call $U(\mathfrak{h})$-free modules, have been extensively studied. Especially, the authors classified the $U(\mathfrak{h})$-free modules of rank one for  $\mathfrak{L}$ in \cite{CH1}. Moreover, the irreducibility and isomorphism classes of these modules were determined therein. However, the theory of representation over $\mathfrak{L}$ is far more from being well-developed.

It is well known that an important way to construct new modules over an algebra is to consider the linear tensor product of known modules
over the algebra (cf.~ \cite{CG2, CHSY}, \cite{GLW}, \cite{TZ1}, \cite{TZ2}, \cite{Z}). As a follow-up of our previous paper \cite{CY}, the purpose of the present paper is to construct new irreducible
non-weight $\mathfrak L$-modules by taking tensor products of a finite number of irreducible modules  defined in \cite{CH1} with irreducible highest weight modules.

This paper is organized as follows. In Section 2, we recall the definitions of various modules involved in this paper and some basic known results on them.  In Section 3, we obtain a class of $\mathfrak L$-modules by taking tensor products of several modules of type $M(\lambda, \alpha, \beta, \gamma)$ and the irreducible highest modules $V(\eta, \epsilon, \theta)$. We also determine the necessary and sufficient conditions for such tensor product modules to be irreducible and study their  submodule structures when they are reducible. Section 4 is devoted to determining the necessary and sufficient conditions for two such irreducible tensor modules to be isomorphic. In Section 5, we compare the irreducible tensor product modules constructed in Section 3 with other known non-weight irreducible modules, and conclude that these irreducible modules provide a class of new irreducible $\mathfrak L$-modules.
\section{Preliminaries and related known results}\label{pre}
We first recall the definitions of some $\mathfrak L$-modules considered in this paper. Denote by $\C[s,t]$ the polynomial algebra in variables $s$ and $t$ with coefficients in $\C$.
\begin{defi}\label{defi2.2} For $\lambda,\a\in\C^*, \b,\r\in\C, i\in\Z$ and $g(s,t)\in\C[s,t]$, define the $\mathfrak{L}$-module action on $\C[s,t]$ as follows:
\begin{align*}
\Omega(\lambda,\a, \b, \r):&\ \ \ \ e_i(g(s,t))=\lambda^i\a g(s-i,t-2),\\
& \ \ \ \ f_i(g(s,t))=-\frac{\lambda^i}\a(\frac{t}{2}-\b)(\frac{t}{2}+\b+1)g(s-i,t+2),\\
& \ \ \ \ h_i(g(s,t))=\lambda^itg(s-i,t),\ \ \ \ d_i(g(s,t))=\lambda^i(s+i\r)g(s-i,t),\\
& \ \ \ \ C(g(s,t))=0;\\
\Delta(\lambda, \a, \b, \r):&\ \ \ \ e_i(g(s,t))=-\frac{\lambda^i}{\a}(\frac{t}{2}+\b)(\frac{t}{2}-\b-1)g(s-i,t-2),\\
& \ \ \ \ f_i(g(s,t))=\lambda^i\a g(s-i,t+2), \ \ \ \ h_i(g(s,t))=\lambda^{i}tg(s-i,t),\\
& \ \ \ \ d_i(g(s,t))=\lambda^i(s+i\r)g(s-i,t),\\
& \ \ \ \ C(g(s,t))=0;\\
\Theta(\lambda, \a, \b, \r):&\ \ \ \ e_i(g(s,t))=\lambda^i\a(\frac{t}{2}+\b)g(s-i,t-2),\\
& \ \ \ \ f_i(g(s,t))=-\frac{\lambda^i}{\a}(\frac{t}{2}-\b)g(s-i,t+2),\\
& \ \ \ \ h_i(g(s,t))={\lambda^i}tg(s-i,t),\ \ \ \ d_i(g(s,t))=\lambda^i(s+i\r)g(s-i,t),\\
& \ \ \ \ C(g(s,t))=0.
\end{align*}
\end{defi}
These modules were introduced in \cite{CH1} to characterize the $U(\mathfrak{h})$-free modules of rank one for  $\mathfrak{L}$.
It is worthwhile to point out that $\C[s,t]$ in each case has the same module structure over the subalgebra $\text{span}\{h_i,d_i,C\mid i\in\Z\}$.
%
%
For later use, we need the following known result on conditions for irreducibility and a classification of isomorphism classes for the modules constructed above.
\begin{prop}\label{pop1}(cf. \cite{CH1})
Keep notations as above, then the following statements hold.
\begin{itemize}
\item[(1)]
$\Omega(\lambda,\a, \b, \r)$ and $\Delta(\lambda, \a, \b, \r)$ are irreducible for any $\lambda,\a\in\C^*$ and $\b,\r\in\C;$ $\Theta(\lambda, \a, \b, \r)$ is irreducible if and only if $2\b \notin \Z_+$.
\item[(2)]
Let $\lambda_1, \lambda_2, \a_1,\a_2\in\mathbb{C}^*,\b_1, \b_2,\r_1, \r_2\in\mathbb{C}$. Then
$$\Omega(\lambda_1,\a_1,\b_1,\r_1), \Delta(\lambda_1,\a_1,\b_1,\r_1), \Theta(\lambda_1,\a_1,\b_1,\r_1)$$ are pairwise non-isomorphic for all parameter choices. Moreover,
\begin{eqnarray*} \label{xxit11}&\Omega(\lambda_1,\a_1,\b_1,\r_1)\cong
\Omega(\lambda_2, \a_2,\b_2,\r_2)\Longleftrightarrow & (\lambda_1,\a_1,\b_1,\r_1)=(\lambda_2,\a_2,\b_2,\r_2)\\
&&{\rm or}\,\, (\lambda_1,\a_1,\b_1,\r_1)=(\lambda_2,\a_2,-\b_{2}-1,\r_2);\nonumber\\
\label{xxit1}&\Delta(\lambda_1,\a_1,\b_1,\r_1)\cong
\Delta(\lambda_2,\a_2,\b_2,\r_2)\Longleftrightarrow & (\lambda_1,\a_1,\b_1,\r_1)=(\lambda_2,\a_2,\b_2,\r_2)\\
&&{\rm or}\,\, (\lambda_1,\a_1,\b_1,\r_1)=(\lambda_2,\a_2,-\b_{2}-1,\r_2);\nonumber\\
\label{xxit2}&\Theta(\lambda_1,\a_1,\b_1,\r_1)\cong\Theta(\lambda_2,\a_2,\b_2,\r_2)\Longleftrightarrow & (\lambda_1,\a_1,\b_1,\r_1)=(\lambda_2,\a_2,\b_2,\r_2).\end{eqnarray*}
\end{itemize}
\end{prop}
For any $\eta, \epsilon, \theta\in\C$, let $I(\eta, \epsilon, \theta)$ be the left ideal of $U(\mathfrak{L})$ generated by the following elements
$$\{e_0, e_{i}, f_{i},  h_{i}, d_{i}\mid i\in\N\}\cup\{d_0-\eta, h_0-\epsilon, C-\theta\}. $$
The Verma $\mathfrak{L}$-module with highest weight $(\eta,  \epsilon, \theta)$ is defined as the quotient module
\begin{equation*}\overline{V}(\eta, \epsilon, \theta)=U(\mathfrak{L})/I(\eta, \epsilon, \theta).\end{equation*}
By the PBW theorem, $\overline{V}(\eta, \epsilon, \theta)$ has a basis  consisting of all vectors of the form
\begin{equation*}f_{-q}^{F_{-q}}\cdots f_0^{F_0}e_{-p}^{E_{-p}}\cdots e_{-1}^{E_{-1}}h_{-m}^{H_{-m}}\cdots h_{-1}^{H_{-1}}d_{-n}^{D_{-n}}\cdots d_{-1}^{D_{-1}} v_h ,\end{equation*}
where $v_h$ is the coset of $1$ in $\overline{V}(\eta, \epsilon, \theta)$, and $$D_{-1},\ldots,D_{-n}, H_{-1}, \ldots, H_{-m}, E_{-1}, \ldots, E_{-p}, F_0,  \ldots, F_{-q}\in\Z_+.$$
Then we have the irreducible  highest weight module $V(\eta, \epsilon, \theta)=\overline{V}(\eta, \epsilon, \theta)/J$, where $J$ is the unique maximal proper submodule of $\overline{V}(\eta, \epsilon, \theta)$. Readers can refer to \cite{B, K} for the structure of $V(\eta, \epsilon, \theta)$.

In the rest of this paper, we will always assume $\lambda,\a\in\C^*, \b,\r,\eta, \epsilon, \theta \in\C$, $M(\lambda, \alpha, \beta, \gamma)=\Omega(\lambda, \alpha, \beta, \gamma)$, $\Delta(\lambda, \alpha, \beta, \gamma)$ or $\Theta(\lambda, \alpha, \beta, \gamma)$  ($2\b \notin \Z_+$) constructed in Definition \ref{defi2.2}, and  $V(\eta, \epsilon, \theta)$ is an  irreducible  highest weight  $\mathfrak{L}$-module.
\section{Irreducibility of tensor product modules}
In this section, we investigate the structures of the tensor products of the irreducible  $\mathfrak L$-modules defined in the previous section, that is, the modules $M(\lambda, \alpha, \beta, \gamma)$ and $V(\eta, \epsilon, \theta)$. In particular, we  determine their irreducibility.

Let $s_1, s_2, \ldots, s_m, t_1, t_2, \ldots, t_m$ be commuting variables, where $m\in\N$. Fix any complex numbers $\lambda_k, \alpha_k\in\C^*, \beta_k, \gamma_k\in\C$, $1\leq k \leq m$, we have the modules  $M(\lambda_k, \alpha_k, \beta_k, \gamma_k)$ defined as in Definition \ref{defi2.2}. Note also that  $2\b_k \notin \Z_+$ when $M=\Theta$. Then we have  $M(\lambda_k, \alpha_k, \beta_k, \gamma_k)=\C[s_k, t_k]$ for $1\leq k \leq m$ as vector spaces. For an irreducible  highest weight module $V(\eta, \epsilon, \theta)$, we can form the tensor product
\begin{equation}\label{a122}\mathbf{T}(\bm{\lambda}, \bm{\alpha}, \bm{\beta}, \bm{\gamma},  \eta, \epsilon, \theta)^{M}=(\otimes_{k=1}^mM(\lambda_k, \alpha_k, \beta_k, \gamma_k))\otimes V(\eta, \epsilon, \theta),\end{equation}
where $\bm{\lambda}=(\lambda_1, \ldots, \lambda_m), \bm{\alpha}=(\alpha_1, \ldots, \alpha_m), \bm{\beta}=(\beta_1, \ldots, \beta_m), \bm{\gamma}=(\gamma_1, \ldots, \gamma_m)$. Denote simply $\mathbf{T}^{M}:=\mathbf{T}(\bm{\lambda}, \bm{\alpha}, \bm{\beta}, \bm{\gamma}, \eta, \epsilon, \theta)^{M}$. Take any nonzero element
\begin{equation*}\label{a12}g(s_1, s_2, \ldots, s_m, t_1, t_2, \ldots, t_m)=\sum_{(\mathbf{p}, \mathbf{q})\in E}s_1^{p_1}t_1^{q_1}\otimes\cdots\otimes s_m^{p_m}t_m^{q_m}\otimes v_{(\mathbf{p}, \mathbf{q})}\in \mathbf{T}^{M},\end{equation*}
where $(\mathbf{p}, \mathbf{q})=(p_1, p_2, \ldots, p_m, q_1, q_2, \ldots, q_m), v_{(\mathbf{p}, \mathbf{q})}\in V(\eta, \epsilon, \theta)\setminus \{0\}$ and $E$ is a finite subset of $\Z_+^{2m}$.  For convenience, we may write
$\mathbf{s}^{\mathbf{p}}\mathbf{t}^{\mathbf{q}}=s_1^{p_1}t_1^{q_1}\otimes\cdots \otimes s_m^{p_m}t_m^{q_m}$
and the element $g(\mathbf{s}, \mathbf{t})\in \mathbf{T}^{M}$ can be rewritten as
\begin{equation}\label{a123}g(\mathbf{s}, \mathbf{t})=\sum_{(\mathbf{p}, \mathbf{q})\in E}\mathbf{s}^{\mathbf{p}}\mathbf{t}^{\mathbf{q}}\otimes v_{(\mathbf{p}, \mathbf{q})}.\end{equation}
Denote $P_k= {\rm max\,}\{p_k\mid (\mathbf{p}, \mathbf{q})\in E\}$ for all $1\leq k \leq m$. We simply write $(\mathbf{s}^{\mathbf{p}}\mathbf{t}^{\mathbf{q}})(\mathbf{s}^{\mathbf{p\prime}}\mathbf{t}^{\mathbf{q\prime}})=\mathbf{s}^{\mathbf{p}+\mathbf{p\prime}}\mathbf{t}^{\mathbf{q}+\mathbf{q\prime}}$ for short in the following.

For any positive integer $l$,  we can define a total order $``\succ"$ on $\Z^l$ by
$$(a_1, \ldots, a_l) \succ (b_1, \ldots, b_l) \Longleftrightarrow  \mathrm{ there\ exists} \ k\in\N \ \mathrm{such \ that}  \ a_k>b_k \ \mathrm{and}\ a_{i}=b_{i},\ \forall\, 1\leq i<k.$$
Then we define the degree ${\rm deg\,}(g(\mathbf{s}, \mathbf{t}))$ of $g(\mathbf{s}, \mathbf{t})$ as the maximal $(\mathbf{p}, \mathbf{q})\in E$ with $v_{(\mathbf{p}, \mathbf{q})}\neq0$. Note that ${\rm deg\,}(1\otimes \cdots \otimes 1\otimes v)=\mathbf{0}=(0, 0,\ldots, 0)$ for $v\in V(\eta, \epsilon, \theta)\setminus \{0\}$.

We need the following two crucial results which will be used repeatedly throughout the paper.
\begin{lemm}\label{prop1}(cf. \cite{TZ2})
Let $\lambda_1, \lambda_2, \ldots, \lambda_m\in\C^*$,  $s_1, s_2, \ldots, s_m\in\N$ with $s_1+s_2+\cdots+s_m=s$. Define a sequence of functions on $\Z$ as follows: $f_1(n)=\lambda_1^n, f_2(n)=n\lambda_1^n, \ldots, f_{s_1}(n)=n^{s_1-1}\lambda_1^n, f_{s_1+1}(n)=\lambda_2^n, \ldots, f_{s_1+s_2}(n)=n^{s_2-1}\lambda_2^n, \ldots, f_{s}(n)=n^{s_m-1}\lambda_m^n$. Let $\mathfrak{M}=(y_{pq})$ be the $s\times s$ matrix with $y_{pq}=f_q(p-1)$, $q=1, 2, \ldots, s, p=r+1, r+2, \ldots, r+s$ where $r\in\Z_+$. Then
\begin{equation*}det(\mathfrak{M})=\prod_{j=1}^m(s_j-1)!!\lambda_j^{s_j(s_j+2r-1)/2}\prod_{1\leq i<j\leq m}(\lambda_j-\lambda_i)^{s_is_j},\end{equation*}
where $s_j!!=s_j!\times (s_j-1)!\times\cdots \times 2!\times 1!$ with $0!!=1$.
\end{lemm}
\begin{prop}\label{prop2}
Suppose that $\lambda_1, \ldots, \lambda_m$ are distinct. Let $W$ be a subspace of $\mathbf{T}^{M}$ which is stable under the action of $h_i,\, d_i$ for any $i$ sufficiently large. Then for any $g(\mathbf{s}, \mathbf{t})=\sum_{(\mathbf{p}, \mathbf{q})\in E}\mathbf{s}^{\mathbf{p}}\mathbf{t}^{\mathbf{q}}\otimes v_{(\mathbf{p}, \mathbf{q})}\in W$,  we have
\begin{eqnarray}\label{111}\sum_{(\mathbf{p}, \mathbf{q})\in E, p_k=P_k}\mathbf{s}^{\mathbf{p}-P_k\omega_k}\mathbf{t}^{\mathbf{q}+\omega_k}\otimes v_{(\mathbf{p}, \mathbf{q})}\in W,\\
\label{222}t_kg(\mathbf{s}, \mathbf{t})=\sum_{(\mathbf{p}, \mathbf{q})\in E}\mathbf{s}^{\mathbf{p}}\mathbf{t}^{\mathbf{q}+\omega_k}\otimes v_{(\mathbf{p}, \mathbf{q})}\in W,\\
\label{3335}s_kg(\mathbf{s}, \mathbf{t})=\sum_{(\mathbf{p}, \mathbf{q})\in E}\mathbf{s}^{\mathbf{p}+\omega_k}\mathbf{t}^{\mathbf{q}}\otimes v_{(\mathbf{p}, \mathbf{q})}\in W,
\end{eqnarray}
where $\omega_k=(\delta_{k, 1}, \delta_{k, 2},\ldots, \delta_{k, m})$ and $k=1, \ldots, m$.\end{prop}
\begin{proof}For sufficiently large integer $i$, we have
\begin{eqnarray}\label{zsa11}
h_i(g(\mathbf{s}, \mathbf{t}))&=&\sum_{(\mathbf{p}, \mathbf{q})\in E}\sum_{k=1}^ms_1^{p_1}t_1^{q_1}\otimes\cdots \otimes h_i(s_k^{p_{k}}t_k^{q_k})\otimes\cdots\otimes s_m^{p_m}t_m^{q_m}\otimes v_{(\mathbf{p}, \mathbf{q})}\nonumber\\
&=& \sum_{(\mathbf{p}, \mathbf{q})\in E}\sum_{k=1}^ms_1^{p_1}t_1^{q_1}\otimes\cdots \otimes \lambda_k^i(s_k-i)^{p_{k}}t_k^{q_k+1}\otimes\cdots\otimes s_m^{p_m}t_m^{q_m}\otimes v_{(\mathbf{p}, \mathbf{q})}\nonumber\\
&=& \sum_{(\mathbf{p}, \mathbf{q})\in E}\sum_{k=1}^m\sum_{x=0}^{p_k}(-1)^x i^x\lambda_k^i s_1^{p_1}t_1^{q_1}\otimes\cdots \otimes \binom{p_k}{x}s_k^{p_{k}-x}t_k^{q_k+1}\otimes\cdots\otimes s_m^{p_m}t_m^{q_m}\otimes v_{(\mathbf{p}, \mathbf{q})}
\end{eqnarray}
and
\begin{eqnarray}\label{zsa2}
d_i(g(\mathbf{s}, \mathbf{t}))&=&\sum_{(\mathbf{p}, \mathbf{q})\in E}\sum_{k=1}^ms_1^{p_1}t_1^{q_1}\otimes\cdots \otimes d_i(s_k^{p_{k}}t_k^{q_k})\otimes\cdots\otimes s_m^{p_m}t_m^{q_m}\otimes v_{(\mathbf{p}, \mathbf{q})}\nonumber\\
&=& \sum_{(\mathbf{p}, \mathbf{q})\in E}\sum_{k=1}^ms_1^{p_1}t_1^{q_1}\otimes\cdots \otimes \lambda_k^i(s_k+i\gamma_k)(s_k-i)^{p_{k}}t_k^{q_k}\otimes\cdots\otimes s_m^{p_m}t_m^{q_m}\otimes v_{(\mathbf{p}, \mathbf{q})}\nonumber\\
&=& \sum_{(\mathbf{p}, \mathbf{q})\in E}\sum_{k=1}^m\sum_{y=0}^{p_k+1}(-1)^y i^y\lambda_k^i s_1^{p_1}t_1^{q_1}\otimes\cdots \otimes (\binom{p_k}{y}-\gamma_k\binom{p_k}{y-1})s_k^{p_{k}-y+1}t_k^{q_k}\nonumber\\
&&     \otimes\cdots\otimes s_m^{p_m}t_m^{q_m}\otimes v_{(\mathbf{p}, \mathbf{q})}.
\end{eqnarray}
where we make the convention that $\binom{0}{0}=1$ and $\binom{p_k}{y}=0$ whenever $y>p_k$ or $y<0$.  Thanks to Lemma \ref{prop1}, we know that the coefficients of $i^x\lambda_k^i$ and $i^y\lambda_k^i$ in \eqref{zsa11} and \eqref{zsa2} belong to $W$ for any $0\leq x \leq P_k, 0\leq y \leq P_k+1$ and $1\leq k \leq m$, respectively. That is,
\begin{equation}\label{zzzc}a_{x, k}:=\sum_{(\mathbf{p}, \mathbf{q})\in E}s_1^{p_1}t_1^{q_1}\otimes\cdots \otimes \binom{p_k}{x}s_k^{p_{k}-x}t_k^{q_k+1}\otimes\cdots\otimes s_m^{p_m}t_m^{q_m}\otimes v_{(\mathbf{p}, \mathbf{q})}\in W,\end{equation}
\begin{equation*}b_{y, k}:=\sum_{(\mathbf{p}, \mathbf{q})\in E}s_1^{p_1}t_1^{q_1}\otimes\cdots \otimes (\binom{p_k}{y}-\gamma_k\binom{p_k}{y-1})s_k^{p_{k}-y+1}t_k^{q_k}\otimes\cdots\otimes s_m^{p_m}t_m^{q_m}\otimes v_{(\mathbf{p}, \mathbf{q})}\in W.
\end{equation*}
For any $1\leq k \leq m$, taking $x=P_k, 0$ and $y= 0$ in the above two elements, respectively, we get \eqref{111}-\eqref{3335}. This completes the proof.
\end{proof}

The following result implies that any element of the form $1\otimes\cdots \otimes 1\otimes v$ generates the whole tensor product module $\mathbf{T}^{M}$ for any $0\neq v\in V(\eta, \epsilon, \theta)$.
\begin{prop}\label{prop3}
Suppose that $\lambda_1, \ldots, \lambda_m$ are distinct. Then $1\otimes\cdots \otimes 1\otimes v$ generates the module $\mathbf{T}^{M}$ for any $0\neq v\in V(\eta, \epsilon, \theta)$.
\end{prop}
\begin{proof}
Fix any nonzero $v\in V(\eta, \epsilon, \theta)$. Let $W$ be the submodule of $\mathbf{T}^{M}$ generated by  $1\otimes\cdots \otimes 1\otimes v$. On one hand, it follows from  \eqref{222} that $\mathbf{t}^{\mathbf{q}}\otimes v\in W$ for all $\mathbf{q}\in\Z_+^m$ by induction on $\mathbf{q}$.
On the other hand, we further deduce from  \eqref{3335} that $\mathbf{s}^{\mathbf{p}}\mathbf{t}^{\mathbf{q}}\otimes v\in W$ for all $(\mathbf{p}, \mathbf{q})\in\Z_+^{2m}$
by induction on $\mathbf{p}$. That is $\C[\mathbf{s}, \mathbf{t}]\otimes v\subseteq W$. Let $V^{\prime}:=\{u\in V(\eta, \epsilon, \theta)\mid \C[\mathbf{s}, \mathbf{t}]\otimes u\subseteq W\}$. The previous argument implies that $V^{\prime}\neq\{0\}$. Note that $V^{\prime}$ is an $\mathfrak L$-submodule of  $V(\eta, \epsilon, \theta)$. It follows that $W=\mathbf{T}^{M}$ by the irreducibility of $V(\eta, \epsilon, \theta)$, as desired.
\end{proof}

We are now in a position to determine a sufficient condition for the tensor product module $\mathbf{T}^{M}$ to be irreducible.

\begin{theo}\label{theoo1}
Let $m\in\N$, $\lambda_k, \alpha_k\in\C^*, \beta_k, \gamma_k, \eta, \epsilon, \theta\in\C$ for $k=1, 2, \ldots, m$ with the $\lambda_k$ pairwise distinct. Then
the tensor product module $\mathbf{T}^{M}$ is irreducible provided  that $2\b_k \notin \Z_+$ for any $1\leq k\leq  m$ when $M=\Theta$.
\end{theo}
\begin{proof}
Let $W$ be a nonzero submodule of $\mathbf{T}^{M}$ and take a nonzero element $g(\mathbf{s}, \mathbf{t})\in W$ with minimal degree. We claim ${\rm deg\,}(g(\mathbf{s}, \mathbf{t}))=\mathbf{0}$ and hence $g(\mathbf{s}, \mathbf{t})=1\otimes\cdots \otimes 1\otimes v$ for some nonzero $v\in V(\eta, \epsilon, \theta)$. Therefore, by Proposition \ref{prop3} we have $W=\mathbf{T}^{M}$ and  $\mathbf{T}^{M}$ is irreducible.

Assume conversely that ${\rm deg\,}(g(\mathbf{s}, \mathbf{t}))\succ \mathbf{0}$. Write $g(\mathbf{s}, \mathbf{t})$ in the form of \eqref{a123}, i.e., $g(\mathbf{s}, \mathbf{t})=\sum_{(\mathbf{p}, \mathbf{q})\in E}\mathbf{s}^{\mathbf{p}}\mathbf{t}^{\mathbf{q}}\otimes v_{(\mathbf{p}, \mathbf{q})}$. We claim that $\mathbf{p}=\mathbf{0}$ for any $(\mathbf{p}, \mathbf{q})\in E$. If this is not true, let $k_0:=\min\{k\mid p_k>0, (\mathbf{p}, \mathbf{q})\in E\ \text{for\ some}\ \mathbf{q}\}$. It follows from \eqref{111} that
$$0\neq\sum_{(\mathbf{p}, \mathbf{q})\in E,\, p_{k_0}=P_{k_0}}\mathbf{s}^{\mathbf{p}-P_{k_0}\omega_{k_0}}\mathbf{t}^{\mathbf{q}+\omega_{k_0}}\otimes v_{(\mathbf{p}, \mathbf{q})}\in W,$$
which has lower degree than $g(\mathbf{s}, \mathbf{t})$. This contradicts the choice of $g(\mathbf{s}, \mathbf{t})$. Thus, the claim follows, i.e., $\mathbf{p}=\mathbf{0}$ for any $(\mathbf{p}, \mathbf{q})\in E$. Then there exists a minimal $j_0$ with $1\leq j_0\leq m$ such that $q_{j_0}>0$. Set $Q_{j_0}=\max\{q_{j_0}\mid q_{j_0}\neq0, (\mathbf{0}, \mathbf{q})\in E\}$. We divide the following discussion into three cases.
\begin{case}
$M=\Omega$.
\end{case}

In this case, a straightforward computation yields that
\begin{eqnarray}\label{zsa1}
e_i(g(\mathbf{s}, \mathbf{t}))&=&\sum_{(\mathbf{0}, \mathbf{q})\in E}\sum_{k=1}^mt_1^{q_1}\otimes\cdots \otimes e_i(t_k^{q_k})\otimes\cdots\otimes t_m^{q_m}\otimes v_{(\mathbf{0}, \mathbf{q})}\nonumber\\
&=& \sum_{(\mathbf{0}, \mathbf{q})\in E}\sum_{k=1}^m\lambda_k^i \alpha_k t_1^{q_1}\otimes\cdots \otimes (t_k-2)^{q_k}\otimes\cdots\otimes t_m^{q_m}\otimes v_{(\mathbf{0}, \mathbf{q})},
\end{eqnarray}
where $i$ is sufficiently large. Applying Lemma \ref{prop1} to the above element, we see that
the coefficient of  $\lambda_{j_0}^i$ lies in $W$, which together with the fact that $\a_{j_0}\neq0$ implies
$$(t_{j_0}-2)^{q_{j_0}}\mathbf{t}^{\mathbf{q}-q_{j_0}\omega_{j_0}}\otimes v_{(\mathbf{0}, \mathbf{q})}\in W.$$
Subtracting the above element from $g(\mathbf{s}, \mathbf{t})$, we have
\begin{eqnarray*}
&&\sum_{(\mathbf{0}, \mathbf{q})\in E}(t_{j_0}^{q_{j_0}}-(t_{j_0}-2)^{q_{j_0}})\mathbf{t}^{\mathbf{q}-q_{j_0}\omega_{j_0}}\otimes v_{(\mathbf{0}, \mathbf{q})}\nonumber\\
&=&\sum_{(\mathbf{0}, \mathbf{q})\in E,
q_{j_0}=Q_{j_0}}(t_{j_0}^{Q_{j_0}}-(t_{j_0}-2)^{Q_{j_0}})\mathbf{t}^{\mathbf{q}-Q_{j_0}\omega_{j_0}}\otimes v_{(\mathbf{0}, \mathbf{q})}+\mbox{lower terms w. r. t. $t_{j_0}$},
\end{eqnarray*}
which is a nonzero element in $W$ and  has lower degree than $g(\mathbf{s}, \mathbf{t})$. This is a contradiction with the choice of $g(\mathbf{s}, \mathbf{t})$. So ${\rm deg\,}(g(\mathbf{s}, \mathbf{t}))=\mathbf{0}$.

\begin{case}$M=\Delta$.\end{case}

Using similar arguments as for  Case 1 but replacing  $e_i$ with $f_i$, we see that ${\rm deg\,}(g(\mathbf{s}, \mathbf{t}))=\mathbf{0}$.

\begin{case}\label{333}
 $M=\Theta$.\end{case}

In this case, for sufficiently large $i$, we have
\begin{eqnarray*}
e_i(g(\mathbf{s}, \mathbf{t}))&=&\sum_{(\mathbf{0}, \mathbf{q})\in E}\sum_{k=1}^mt_1^{q_1}\otimes\cdots \otimes e_i(t_k^{q_k})\otimes\cdots\otimes t_m^{q_m}\otimes v_{(\mathbf{0}, \mathbf{q})}\nonumber\\
&=& \sum_{(\mathbf{0}, \mathbf{q})\in E}\sum_{k=1}^m\lambda_k^i \alpha_k t_1^{q_1}\otimes\cdots \otimes (\frac{t_k}{2}+\b_k)(t_k-2)^{q_k}\otimes\cdots\otimes t_m^{q_m}\otimes v_{(\mathbf{0}, \mathbf{q})},\\
f_i(g(\mathbf{s}, \mathbf{t}))&=&\sum_{(\mathbf{0}, \mathbf{q})\in E}\sum_{k=1}^mt_1^{q_1}\otimes\cdots \otimes f_i(t_k^{q_k})\otimes\cdots\otimes t_m^{q_m}\otimes v_{(\mathbf{0}, \mathbf{q})}\nonumber\\
&=& -\sum_{(\mathbf{0}, \mathbf{q})\in E}\sum_{k=1}^m\frac{\lambda_k^i} {\alpha_k} t_1^{q_1}\otimes\cdots \otimes (\frac{t_k}{2}-\b_k)(t_k+2)^{q_k}\otimes\cdots\otimes t_m^{q_m}\otimes v_{(\mathbf{0}, \mathbf{q})}
\end{eqnarray*}
and
\begin{eqnarray*}
h_i(g(\mathbf{s}, \mathbf{t}))&=&\sum_{(\mathbf{0}, \mathbf{q})\in E}\sum_{k=1}^mt_1^{q_1}\otimes\cdots \otimes h_i(t_k^{q_k})\otimes\cdots\otimes t_m^{q_m}\otimes v_{(\mathbf{0}, \mathbf{q})}\nonumber\\
&=& \sum_{(\mathbf{0}, \mathbf{q})\in E}\sum_{k=1}^m\lambda_k^i  t_1^{q_1}\otimes\cdots \otimes t_k^{q_k+1}\otimes\cdots\otimes t_m^{q_m}\otimes v_{(\mathbf{0}, \mathbf{q})}.
\end{eqnarray*}
Applying  Lemma \ref{prop1} to the above three elements, respectively,  we know that the coefficients of $\lambda_{j_0}^i$ belong to $W$, i.e.,
\begin{eqnarray*}0\neq \alpha_{j_0}g_{1, {j_0}}:=\sum_{(\mathbf{0}, \mathbf{q})\in E}\alpha_{j_0}(\frac{t_{j_0}}{2}+\b_{j_0})(t_{j_0}-2)^{q_{j_0}}\mathbf{t}^{\mathbf{q}-q_{j_0}\omega_{j_0}}\otimes v_{(\mathbf{0}, \mathbf{q})}\in W,\\
0\neq -\frac{1}{\alpha_{j_0}}g_{2, {j_0}}:=-\sum_{(\mathbf{0}, \mathbf{q})\in E}\frac{1}{\alpha_{j_0}}(\frac{t_{j_0}}{2}-\b_{j_0})(t_{j_0}+2)^{q_{j_0}}\mathbf{t}^{\mathbf{q}-q_{j_0}\omega_{j_0}}\otimes v_{(\mathbf{0}, \mathbf{q})}\in W,\\
0\neq g_{3, {j_0}}:=\sum_{(\mathbf{0}, \mathbf{q})\in E}t_{j_0}^{q_{j_0}+1}\mathbf{t}^{\mathbf{q}-q_{j_0}\omega_{j_0}}\otimes v_{(\mathbf{0}, \mathbf{q})}\in W.
\end{eqnarray*}
Considering a linear combination of  $g_{1, j_{0}}, g_{2, j_{0}}, g_{3, j_{0}}$, we have
\begin{eqnarray*}
&&g_{1, j_{0}}+g_{2, j_{0}}-g_{3, j_{0}}\nonumber\\
&=&\sum_{(\mathbf{0}, \mathbf{q})\in E}\big((\frac{t_{j_0}}{2}+\b_{j_0})(t_{j_0}-2)^{q_{j_0}}+(\frac{t_{j_0}}{2}-\b_{j_0})(t_{j_0}+2)^{q_{j_0}}-t_{j_0}^{q_{j_0}+1}\big)\mathbf{t}^{\mathbf{q}-q_{j_0}\omega_{j_0}}\otimes v_{(\mathbf{0}, \mathbf{q})}\\
&=&\sum_{(\mathbf{0}, \mathbf{q})\in E,
q_{j_0}=Q_{j_0}}2Q_{j_0}(Q_{j_0}-1-2\b_{j_0})t_{j_0}^{Q_{j_0}-1}\mathbf{t}^{\mathbf{q}-Q_{j_0}\omega_{j_0}}\otimes v_{(\mathbf{0}, \mathbf{q})}\nonumber\\
&&+\mbox{lower terms w. r. t. $t_{j_0}$},
\end{eqnarray*}
which is a nonzero element in $W$ and has lower degree than $g(\mathbf{s}, \mathbf{t})$. Hence, ${\rm deg\,}(g(\mathbf{s}, \mathbf{t}))=\mathbf{0}$.  We complete the proof.
\end{proof}

Now we consider the case when $\lambda_1, \ldots, \lambda_m$ are not distinct. Actually,  $\mathbf{T}^{M}$ is reducible in this case and it suffices to show the reducibility of $M(\lambda, \alpha_1, \beta_1, \gamma_1)\otimes M(\lambda, \alpha_2, \beta_2, \gamma_2)$. We still use the notations as before, only taking $m=2$ and $V(\eta, \epsilon, \theta)$ as the $1$-dimensional trivial module. For convenience, we identify $M(\lambda, \alpha_1, \beta_1, \gamma_1)\otimes M(\lambda, \alpha_2, \beta_2, \gamma_2)=\C[s_1, s_2, t_1, t_2]$. For any $l\in\Z_+$, denote
$$W_l={\rm span\,}\{s_1^r(s_1+s_2)^p\C[t_1, t_2]\mid r, p\in\Z_+, r\leq l\}.$$
Clearly, $W_l\subset W_{l+1}$. And, we have the following
\begin{theo}\label{theo2}
Keep notations as above, then each $W_l$ is a proper submodule of $M(\lambda, \alpha_1, \beta_1, \gamma_1)\otimes M(\lambda, \alpha_2, \beta_2, \gamma_2)$.
\end{theo}
\begin{proof}
We only show the assertion for the case  $M=\Omega$,  the other two cases can be treated similarly. Fix any $l\in\Z_+$. For any $s_1^r(s_1+s_2)^pg_1(t_1)g_2(t_2)\in W_l$, where $r, p\in\Z_+, r\leq l, g_1(t_1)\in\C[t_1], g_2(t_2)\in\C[t_2]$, we can compute
\begin{eqnarray*}
&&\lambda^{-i}d_i\big(s_1^r(s_1+s_2)^pg_1(t_1)g_2(t_2)\big)\\
&=&\lambda^{-i}d_i\big(\sum_{j=0}^p\binom{p}{j}s_1^{r+j}s_2^{p-j}g_1(t_1)g_2(t_2)\big)\\
&=& \lambda^{-i}\sum_{j=0}^p\binom{p}{j}\big(d_i(s_1^{r+j}g_1(t_1))s_2^{p-j}g_2(t_2)
+s_1^{r+j}g_1(t_1)d_i(s_2^{p-j}g_2(t_2))\big)\\
&=& (s_1-i)^r(s_1+i\gamma_1)g_1(t_1)g_2(t_2)\sum_{j=0}^p\binom{p}{j}(s_1-i)^js_2^{p-j}\\
&&+ s_1^r(s_2+i\gamma_2)g_1(t_1)g_2(t_2)\sum_{j=0}^p\binom{p}{j}(s_2-i)^{p-j}s_1^j\\
&=&(s_1-i)^r(s_1+i\gamma_1)g_1(t_1)g_2(t_2)(s_1+s_2-i)^p+s_1^r(s_2+i\gamma_2)g_1(t_1)g_2(t_2)(s_1+s_2-i)^p\\
&=&\big(s_1((s_1-i)^r-s_1^r)+i\gamma_1(s_1-i)^r\big)g_1(t_1)g_2(t_2)(s_1+s_2-i)^p\\
&&+s_1^r(s_1+s_2+i\gamma_2)g_1(t_1)g_2(t_2)(s_1+s_2-i)^p,
\end{eqnarray*}
\begin{eqnarray*}
&&\lambda^{-i}h_i\big(s_1^r(s_1+s_2)^pg_1(t_1)g_2(t_2)\big)\\
&=& \lambda^{-i}\sum_{j=0}^p\binom{p}{j}\big(h_i(s_1^{r+j}g_1(t_1))s_2^{p-j}g_2(t_2)
+s_1^{r+j}g_1(t_1)h_i(s_2^{p-j}g_2(t_2))\big)\\
&=&(s_1-i)^rt_1g_1(t_1)g_2(t_2)\sum_{j=0}^p\binom{p}{j}(s_1-i)^js_2^{p-j}\\
&&+ s_1^rg_1(t_1)t_2g_2(t_2)\sum_{j=0}^p\binom{p}{j}(s_2-i)^{p-j}s_1^j\\
&=&(s_1-i)^rt_1g_1(t_1)g_2(t_2)(s_1+s_2-i)^p+s_1^rg_1(t_1)t_2g_2(t_2)(s_1+s_2-i)^p,
\end{eqnarray*}
\begin{eqnarray*}
&&\lambda^{-i}e_i\big(s_1^r(s_1+s_2)^pg_1(t_1)g_2(t_2)\big)\\
&=& \lambda^{-i}\sum_{j=0}^p\binom{p}{j}\big(e_i\cdot(s_1^{r+j}g_1(t_1))s_2^{p-j}g_2(t_2)
+s_1^{r+j}g_1(t_1)e_i(s_2^{p-j}g_2(t_2))\big)\\
&=&\alpha_1(s_1-i)^rg_1(t_1-2)g_2(t_2)\sum_{j=0}^p\binom{p}{j}(s_1-i)^js_2^{p-j}\\
&&+ \alpha_2s_1^rg_1(t_1)g_2(t_2-2)\sum_{j=0}^p\binom{p}{j}(s_2-i)^{p-j}s_1^j\\
&=&\alpha_1(s_1-i)^rg_1(t_1-2)g_2(t_2)(s_1+s_2-i)^p+\alpha_2s_1^rg_1(t_1)g_2(t_2-2)(s_1+s_2-i)^p
\end{eqnarray*}
and
\begin{eqnarray*}
&&\lambda^{-i}f_i\big(s_1^r(s_1+s_2)^pg_1(t_1)g_2(t_2)\big)\\
&=& \lambda^{-i}\sum_{j=0}^p\binom{p}{j}\big(f_i(s_1^{r+j}g_1(t_1))s_2^{p-j}g_2(t_2)
+s_1^{r+j}g_1(t_1)f_i(s_2^{p-j}g_2(t_2))\big)\\
&=&-\frac{1}{\alpha_1}(s_1-i)^r(\frac{t_1}{2}-\beta_1)(\frac{t_1}{2}+\beta_1+1)g_1(t_1+2)g_2(t_2)\sum_{j=0}^p\binom{p}{j}(s_1-i)^js_2^{p-j}\\
&&-\frac{1}{\alpha_2}s_1^rg_1(t_1)(\frac{t_2}{2}-\beta_2)(\frac{t_2}{2}+\beta_2+1)g_2(t_2+2)\sum_{j=0}^p\binom{p}{j}(s_2-i)^{p-j}s_1^j\\
&=&-\frac{1}{\alpha_1}(s_1-i)^r(\frac{t_1}{2}-\beta_1)(\frac{t_1}{2}+\beta_1+1)g_1(t_1+2)g_2(t_2)(s_1+s_2-i)^p\\
&&-\frac{1}{\alpha_2}s_1^rg_1(t_1)(\frac{t_2}{2}-\beta_2)(\frac{t_2}{2}+\beta_2+1)g_2(t_2+2)(s_1+s_2-i)^p,
\end{eqnarray*}
which imply that $$X_i\big(s_1^r(s_1+s_2)^pg_1(t_1)g_2(t_2)\big)\in W_l, \quad{\rm where}\,\,X_i\in\{d_i, h_i, e_i, f_i\mid\forall i\in\Z\}.$$
This together with $CW_l=0$ implies that $\mathfrak LW_l\subset W_l$. So $W_l$ is a submodule of $\Omega(\lambda, \alpha_1, \beta_1, \gamma_1)\otimes \Omega(\lambda, \alpha_2, \beta_2, \gamma_2)$. We complete the proof.
\end{proof}

As a direct consequence of Theorem \ref{theoo1} and Theorem \ref{theo2}, we have the following sufficient and necessary condition for a tensor product module to be irreducible.

\begin{coro}\label{theo3}
The $\mathfrak L$-module  $\mathbf{T}^{M}$ is irreducible if and only if $\lambda_1, \ldots, \lambda_m$ are pairwise distinct.
\end{coro}

\section{Isomorphism classes of the tensor product modules}
In this section, we will determine the necessary and sufficient conditions for two irreducible tensor product modules  $\mathbf{T}^{M}$ to be isomorphic. By Corollary \ref{theo3}, to ensure that $\mathbf{T}^{M}$ is irreducible, we suppose that $\lambda_k, \alpha_k\in\C^*, \beta_k, \gamma_k \in\C$ for $k=1, 2, \ldots, m$ with the $\lambda_k$ pairwise distinct and  $2\b_k \notin \Z_+$ when $M=\Theta$ throughout this section.

For any $g:=g(\mathbf{s}, \mathbf{t})\in \mathbf{T}^{M}$, we define
\begin{equation*}
R_{g}=\left\{\begin{array}{llll}
\lim_{l\rightarrow \infty}{\rm rank\,}\{g, h_i(g), e_i(g)\mid i\geq l\},&\mbox{if \ }g\in \mathbf{T}^{\Omega},\\[4pt]
\lim_{l\rightarrow \infty}{\rm rank\,}\{g, h_i(g), f_i(g)\mid i\geq l\},&\mbox{if \ }g\in \mathbf{T}^{\Delta},\\[4pt]
\lim_{l\rightarrow \infty}{\rm rank\,}\{g, h_i(g), e_i(g), f_i(g)\mid i\geq l\},&\mbox{if \ }g\in \mathbf{T}^{\Theta},
\end{array}\right.\end{equation*}
and
$$R_{\mathbf{T}^{M}}={\rm inf\,}\{R_{g}\mid 0\neq g\in \mathbf{T}^{M}\},$$
where ${\rm rank\,}(Y)={\rm dim\,}{\rm span\,}(Y)$ for any subset $Y$ in a vector space. If we write  $g$ in the form \eqref{a123}, there exists a minimal positive integer $I(g)$ such that $h_iv_{(\mathbf{p}, \mathbf{q})}=e_iv_{(\mathbf{p}, \mathbf{q})}=f_iv_{(\mathbf{p}, \mathbf{q})}=0$ for all $i\geq I(g)$ and $(\mathbf{p}, \mathbf{q})\in E$.

We have the following result describing the property of the invariants $R_{g}$ and $R_{\mathbf{T}^{M}}$.
\begin{lemm}\label{theoo5}
For any nonzero $g\in \mathbf{T}^{M}$, the following statements hold.
\begin{itemize}
  \item[(1)]For all $l\geq I(g)$,
  \begin{equation*}
R_{g}=\left\{\begin{array}{llll}
{\rm rank\,}\{g, h_i(g), e_i(g)\mid i\geq l\},&\mbox{if \ }g\in \mathbf{T}^{\Omega},\\[4pt]
{\rm rank\,}\{g, h_i(g), f_i(g)\mid i\geq l\},&\mbox{if \ }g\in \mathbf{T}^{\Delta},\\[4pt]
{\rm rank\,}\{g, h_i(g), e_i(g), f_i(g)\mid i\geq l\},&\mbox{if \ }g\in \mathbf{T}^{\Theta}.
\end{array}\right.\end{equation*}
  \item[(2)]$R_{g}\geq m+1$ and the equality holds if and only if  $g=1\otimes\cdots \otimes 1 \otimes v$ for some $0\neq v\in V(\eta, \epsilon, \theta)$.
  \item[(3)]$R_{\mathbf{T}^{M}}=m+1$.
\end{itemize}
\end{lemm}
\begin{proof}
(1) we only tackle the case $g\in \mathbf{T}^{\Omega}$, since a similar argument can be applied to the other two cases.
Denote $R_{g, l}= {\rm rank\,}\{g, h_i(g), e_i(g)\mid i\geq l\}$ for any $l\in\N$, then it suffices to show that $R_{g, l}=R_{g, I(g)}$ for all $l\geq I(g)$. For any $i\geq l\geq I(g)$, an explicit calculation yields that
\begin{eqnarray*}
e_i (g)&=&\sum_{(\mathbf{p}, \mathbf{q})\in E}\sum_{k=1}^ms_1^{p_1}t_1^{q_1}\otimes\cdots \otimes e_i(s_k^{p_{k}}t_k^{q_k})\otimes\cdots\otimes s_m^{p_m}t_m^{q_m}\otimes v_{(\mathbf{p}, \mathbf{q})}\nonumber\\
&=& \sum_{(\mathbf{p}, \mathbf{q})\in E}\sum_{k=1}^ms_1^{p_1}t_1^{q_1}\otimes\cdots \otimes \lambda_k^i\alpha_k(s_k-i)^{p_{k}}(t_k-2)^{q_k}\otimes\cdots\otimes s_m^{p_m}t_m^{q_m}\otimes v_{(\mathbf{p}, \mathbf{q})}\nonumber\\
&=& \sum_{(\mathbf{p}, \mathbf{q})\in E}\sum_{k=1}^m\sum_{z=0}^{p_k}(-1)^z i^z\lambda_k^i\alpha_k s_1^{p_1}t_1^{q_1}\otimes\cdots \otimes \binom{p_k}{z}s_k^{p_{k}-z}(t_k-2)^{q_k}\otimes\nonumber\\
&&     \cdots\otimes s_m^{p_m}t_m^{q_m}\otimes v_{(\mathbf{p}, \mathbf{q})}.
\end{eqnarray*}
Considering the coefficient of $i^z\lambda_k^i$ for any $0\leq z \leq P_k$ and $1\leq k \leq m$ in the above element and noting that  $\a_k\neq0$, we get
\begin{equation}\label{ccv}
c_{z, k}:=\sum_{(\mathbf{p}, \mathbf{q})\in E}s_1^{p_1}t_1^{q_1}\otimes\cdots \otimes \binom{p_k}{z}s_k^{p_{k}-z}(t_k-2)^{q_k}\otimes\cdots\otimes s_m^{p_m}t_m^{q_m}\otimes v_{(\mathbf{p}, \mathbf{q})}\in W,
\end{equation}
which is independent of $l$ provided that $l\geq I(g)$. This along with \eqref{zzzc} gives
$${\rm span\,}\{g, h_i(g), e_i(g)\mid i\geq l\}={\rm span\,}\{g, a_{x, k}, c_{z, k}\mid 0\leq x, z \leq P_k, 1\leq k \leq m\}.$$
Consequently, $R_{g, l}=R_{g, I(g)}$ for all $l\geq I(g)$, proving (1).

(2) We first consider the case $g\in\mathbf{T}^{\Omega}$, and the proof for $g\in \mathbf{T}^{\Delta}$ is similar. Combining Proposition \ref{prop2} \eqref{111}, \eqref{222} with \eqref{ccv} , we see that
\begin{eqnarray*}a_{P_k, k}=\sum_{(\mathbf{p}, \mathbf{q})\in E, p_k=P_k}\mathbf{s}^{\mathbf{p}-P_k\omega_k}\mathbf{t}^{\mathbf{q}+\omega_k}\otimes v_{(\mathbf{p}, \mathbf{q})}\in W,\\
a_{0, k}=t_kg(\mathbf{s}, \mathbf{t})=\sum_{(\mathbf{p}, \mathbf{q})\in E}\mathbf{s}^{\mathbf{p}}\mathbf{t}^{\mathbf{q}+\omega_k}\otimes v_{(\mathbf{p}, \mathbf{q})}\in W,\\
c_{0, k}=\sum_{(\mathbf{p}, \mathbf{q})\in E}(t_k-2)^{q_k}\mathbf{s}^{\mathbf{p}}\mathbf{t}^{\mathbf{q}-\omega_k}\otimes v_{(\mathbf{p}, \mathbf{q})}\in W.
\end{eqnarray*}
If ${\rm deg\,}(g)=\mathbf{0}$, that is, $g=1\otimes\cdots \otimes1 \otimes v$ for some $0\neq v\in V(\eta, \epsilon, \theta)$, then we have $R_{g}={\rm rank\,}\{g, a_{0, k}, c_{0, k}\mid 1\leq k\leq m\}=m+1$. Now suppose ${\rm deg\,}(g)=(\mathbf{p}^\prime, \mathbf{q}^\prime)\succ\mathbf{0}$. It is clear that ${\rm deg\,}(a_{0, k})=(\mathbf{p}^\prime, \mathbf{q}^\prime+\omega_k)$ for $ 1\leq k\leq m$. If there exists  $1\leq k\leq m$ such that $p_{k}^{\prime}>0$, let $k_1$ be the minimal $k$ with $p_k^{\prime}>0$. Then we have ${\rm deg\,}(a_{P_{k_1}, k_1})\prec (\mathbf{p}^\prime, \mathbf{q}^\prime)$.
Hence the space spanned by $g, a_{0, k}, a_{P_{k_1}, k_1}, 1\leq k\leq m$ has dimension $m+2$, which means $R_{g}\geq m+2$. If $p_{k}=0$ for all $1\leq k\leq m$, then there  exists  $1\leq k\leq m$ such that $q_{k}^{\prime}>0$. Let $k_2$ be the minimal $k$ with $q_{k}^{\prime}>0$. One can observe that $\mathbf{0}\preceq{\rm deg\,}(g-c_{0, k_2})\prec (\mathbf{p}^\prime, \mathbf{q}^\prime)$. Thus the space spanned by $g, a_{0, k}, c_{0, k_2}, 1\leq k\leq m$ has dimension $m+2$, so that $R_{g}\geq m+2$. Thus, (2) holds for $M\in\{\Omega, \Delta\}$.

For the remaining case $g\in\mathbf{T}^{\Theta}$, by a similar argument in the preceding paragraph, we see that if ${\rm deg\,}(g)=\mathbf{0}$, then $R_{g}=m+1$. If ${\rm deg\,}(g)=(\mathbf{p}^\prime, \mathbf{q}^\prime)\succ\mathbf{0}$ and there exists  $1\leq k\leq m$ such that $p_{k}>0$, let $k_3$ be the minimal $k$ with $p_k^{\prime}>0$. Then ${\rm dim\,}{\rm span\,}\{g, a_{0, k}, a_{P_{k_3}, k_3}\mid1\leq k\leq m\}=m+2$. So, $R_{g}\geq m+2$.  If $p_{k}=0$ for all $1\leq k\leq m$ and there  exists  $1\leq k\leq m$ such that $q_{k}>0$, let $k_4$ be the minimal one. It follows from the computation in  Theorem                                                                                                                                                                       \ref{theoo1} Case \ref{333} that $\mathbf{0}\preceq{\rm deg\,}(g_{1, k_{4}}+g_{2, k_{4}}-g_{3, k_{4}})\prec (\mathbf{p}^\prime, \mathbf{q}^\prime)$, forcing $R_{g}\geq m+2$. Hence, (2) holds for $M=\Theta$, proving  (2).

(3) is an obvious consequence of (2).

We complete the proof.
\end{proof}
Now we are ready to prove our isomorphism theorem. Let $\mathbf{T}^{M}$ be the tensor module defined before. Now we take another tensor module, \begin{equation*}\mathbf{T}'^{M}:=\mathbf{T}'(\bm{\lambda}', \bm{\alpha}', \bm{\beta}', \bm{\gamma}', \eta', \epsilon', \theta')^{M}=(\otimes_{j=1}^{m'}M(\lambda_j', \alpha_j', \beta_j', \gamma_j'))\otimes V(\eta', \epsilon', \theta'),\end{equation*}
where $\bm{\lambda}'=(\lambda_1', \ldots, \lambda_{m'}'), \bm{\alpha}'=(\alpha_1', \ldots, \alpha_{m'}'), \bm{\beta}'=(\beta_1', \ldots, \beta_{m'}'), \bm{\gamma}'=(\gamma_1', \ldots, \gamma_{m'}')$ and $V(\eta', \epsilon', \theta')$ is an irreducible highest weight module. We identify $M(\lambda_j', \alpha_j', \beta_j', \gamma_j')=\C[s_j', t_j']$ for $1\leq j\leq m'$. To ensure that $\mathbf{T}^{M}$ and $\mathbf{T}'^{M}$ are irreducible, we suppose that $\lambda_1, \ldots, \lambda_{m}$  are pairwise distinct as well as  $\lambda_1', \ldots, \lambda_{m'}'$ are pairwise distinct, $\alpha_k, \alpha_{j}'\in\mathbb{C}^*$ for $1\leq k\leq m$ and $1\leq j\leq m'$ and $2\b_k, 2\b_j'\not\in\Z_+$ when $M=\Theta$.
\begin{theo}\label{theoo2}
$\mathbf{T}^{M}\cong\mathbf{T}'^{M}$ as $\mathfrak L$-modules if and only if $m=m', V(\eta,  \epsilon, \theta)\cong V(\eta',  \epsilon', \theta')$
and $M(\lambda_k,\a_k,\b_k,\r_k)\cong
M(\lambda_k', \a_k',\b_k',\r_k')$ for $1\leq k\leq m$ after renumbering the indices $(\lambda_k', \alpha_k', \beta_k', \gamma_k')$ if necessary.
\end{theo}
\begin{proof}
The sufficiency is obvious and it suffices to show the necessity. Let $\phi$ be an
$\mathfrak L$-module isomorphism from $\mathbf{T}^{M}$ to $\mathbf{T}'^{M}$. From Lemma \ref{theoo5} (3), we see that $m+1=R_{\mathbf{T}^{M}}=R_{\mathbf{T}'^{M}}=m'+1$, i.e., $m=m'$.

Take $g=1\otimes\cdots \otimes1 \otimes v$ for some $0\neq v\in V(\eta, \epsilon, \theta)$. From Lemma \ref{theoo5} (2), we have $R_{\phi(g)}=R_{g}=m+1$, yielding
\begin{equation}\label{vvbn}
\phi(g)=1\otimes\cdots \otimes1 \otimes v' \quad{\rm for \,\,some}\,\,  0\neq v'\in V(\eta', \epsilon', \theta'). \end{equation} Since
$v'$ is uniquely determined by $v$, we may denote $\tau(v)=v'$.  It is obvious that $\tau$ is a linear bijection from $V(\eta, \epsilon, \theta)$ to $V(\eta', \epsilon', \theta')$.
We only tackle  the case  $M=\Omega$ in the following, since the other two cases can be treated similarly.

Now for any $i\in\Z$ with $i\geq {\rm max\,}\{I(g), I(\phi(g))\}$, by \eqref{vvbn}, we have
\begin{eqnarray}\label{for44}
0&=&\phi(d_i(1\otimes\cdots \otimes1 \otimes v))-d_i\phi(1\otimes\cdots \otimes1 \otimes v)\nonumber\\
&=&\sum_{k=1}^m\lambda_k^i\phi(1\otimes\cdots \otimes s_k\otimes\cdots \otimes 1\otimes v)+\sum_{k=1}^mi\lambda_k^i\gamma_k(1\otimes\cdots \otimes1 \otimes \tau(v))\nonumber\\
&&-\sum_{k=1}^m(\lambda_k')^i(1\otimes\cdots \otimes s_k'\otimes\cdots \otimes 1\otimes\tau(v))-\sum_{k=1}^mi(\lambda_k')^i\gamma_k'(1\otimes\cdots \otimes1 \otimes \tau(v)).
\end{eqnarray}
From Lemma \ref{prop1}, we know that for any $1\leq k\leq m$, if $\lambda_k\neq \lambda_j'$ for any $1\leq j\leq m$, then the coefficient of $\lambda_k^i$ in \eqref{for44} should be zero, i.e.,  $\phi(1\otimes\cdots \otimes s_k\otimes\cdots \otimes1\otimes v)=0$, which is absurd. So, we may assume that
$\lambda_k= \lambda_k'$ for all $1\leq k\leq m$ by reordering the indices $1\leq k\leq m$. Now the coefficients of $\lambda_k^i$ and $i\lambda_k^i, 1\leq k\leq m$ in \eqref{for44} become $\phi(1\otimes\cdots \otimes s_k\otimes\cdots \otimes1\otimes v)-1\otimes\cdots \otimes s_k'\otimes\cdots \otimes1\otimes \tau(v)$ and $(\gamma_k-\gamma_k')1\otimes1\otimes\cdots \otimes 1 \otimes \tau(v)$, respectively, forcing
\begin{equation}\label{vvb}\phi(1\otimes\cdots \otimes s_k\otimes\cdots \otimes1\otimes v)=1\otimes\cdots \otimes s_k'\otimes\cdots \otimes1\otimes \tau(v)\end{equation} and $\gamma_k=\gamma_k'$ for $1\leq k\leq m$. Now for any $i\geq {\rm max\,}\{I(g), I(\phi(g))\}$, the equations
\begin{eqnarray*}
&\phi(e_i (1\otimes\cdots \otimes1 \otimes v))-e_i (\phi(1\otimes\cdots \otimes1 \otimes v))=0, \\
&\phi(h_i(1\otimes\cdots \otimes1 \otimes v))-h_i (\phi(1\otimes\cdots \otimes1 \otimes v))=0,\\
&\phi(f_i (1\otimes\cdots \otimes1 \otimes v))-f_i (\phi(1\otimes\cdots \otimes1 \otimes v))=0
\end{eqnarray*}
are respectively equivalent to
\begin{eqnarray*}
&\sum_{k=1}^m\lambda_k^i(\alpha_k-\alpha_k')(1\otimes1\otimes\cdots \otimes1 \otimes \tau(v))=0, \\
&\sum_{k=1}^m\lambda_k^i\big(\phi(1\otimes\cdots \otimes t_k\otimes\cdots \otimes1\otimes v)-1\otimes\cdots \otimes t_k'\otimes\cdots 1\otimes \tau(v)\big)=0,\\
&-\sum_{k=1}^m\frac{\lambda_k^i}{\alpha_k}\phi(1\otimes\cdots \otimes (\frac{t_k^2}{4}+\frac{t_k}{2}-\beta_k(\beta_k+1))\otimes\cdots \otimes1\otimes v)\\
&+\sum_{k=1}^m\frac{\lambda_k^i}{\alpha_k'}(1\otimes\cdots \otimes (\frac{t_k'^2}{4}+\frac{t_k'}{2}-\beta_k'(\beta_k'+1))\otimes\cdots \otimes1\otimes \tau(v))=0.
\end{eqnarray*}
From  Lemma \ref{prop1}, we have for any $1\leq k\leq m$,
\begin{eqnarray}
\label{equ11}&\alpha_k=\alpha_k', \\
\label{equ21}&\phi(1\otimes\cdots \otimes t_k\otimes\cdots \otimes1\otimes v)=1\otimes\cdots \otimes t_k'\otimes\cdots \otimes1\otimes \tau(v),\\
&\label{equ31}\frac{1}{\alpha_k}\phi(1\otimes\cdots \otimes (\frac{t_k^2}{4}+\frac{t_k}{2}-\beta_k(\beta_k+1))\otimes\cdots \otimes1\otimes v)\nonumber\\
&=\frac{1}{\alpha_k'}\big(1\otimes\cdots \otimes (\frac{t_k'^2}{4}+\frac{t_k'}{2}-\beta_k'(\beta_k'+1))\otimes\cdots \otimes1\otimes \tau(v)\big).
\end{eqnarray}
For any  $1\leq k\leq m$ and $i\geq {\rm max\,}\{I(g), I(\phi(g))\}$, from \eqref{equ21} and
\begin{equation*}
\phi\big(h_i(1\otimes\cdots \otimes t_k\otimes\cdots \otimes1\otimes v)\big)-h_i\phi(1\otimes\cdots \otimes t_k\otimes\cdots \otimes1\otimes v)=0,
\end{equation*}
we have
\begin{eqnarray*}
&\sum_{j\in\{1,\ldots,m\}\setminus\{k\}}\lambda_j^i\phi(1\otimes\cdots \otimes t_j\otimes\cdots\otimes t_k\otimes\cdots \otimes1\otimes v)\\
&+\lambda_k^i\phi(1\otimes\cdots \otimes t_k^2\otimes\cdots \otimes1\otimes v)\\
&-\sum_{j\in\{1,\ldots,m\}\setminus\{k\}}\lambda_j^i(1\otimes\cdots \otimes t_j'\otimes\cdots\otimes t_k'\otimes\cdots \otimes1\otimes \tau(v))\\
&+\lambda_k^i(1\otimes\cdots \otimes t_k'^2\otimes\cdots \otimes1\otimes \tau(v))=0,
\end{eqnarray*}
Also, it follows from Lemma \ref{prop1} that the coefficient of $\lambda_k^i$ should be zero, that is,
\begin{equation}\label{equ4}
\phi(1\otimes\cdots \otimes t_k^2\otimes\cdots \otimes1\otimes v)=1\otimes\cdots \otimes t_k'^2\otimes\cdots \otimes1\otimes \tau(v).\end{equation}
Applying \eqref{equ11}, \eqref{equ21}, \eqref{equ4} to \eqref{equ31} gives $\beta_k=\beta_k'$ or $\beta_k=-\beta_k'-1$ for any $1\leq k\leq m$. Hence, $\Omega(\lambda_k,\a_k,\b_k,\r_k)\cong
\Omega(\lambda_k', \a_k',\b_k',\r_k')$ by Proposition \ref{pop1} (2).
Combining \eqref{vvbn}, \eqref{vvb},  \eqref{equ21}  with  \eqref{equ4}, we obtain
\begin{equation*}
\phi(X_i (1\otimes\cdots \otimes1)\otimes v)=X_i(1\otimes\cdots \otimes1 )\otimes\tau(v), \quad{\rm where}\,\,X_i\in\{d_i, h_i, e_i, f_i\mid\forall i\in\Z\}.
\end{equation*}
This together with
$$\phi(X_i (1\otimes\cdots \otimes1\otimes v))=X_i (\phi(1\otimes\cdots \otimes1\otimes v)),\quad{\rm where}\,\,X_i\in\{d_i, h_i, e_i, f_i\mid\forall i\in\Z\}$$
gives
$$\phi(1\otimes\cdots \otimes1\otimes X_i( v))=1\otimes\cdots \otimes1\otimes X_i(\tau(v)),\quad{\rm where}\,\,X_i\in\{d_i, h_i, e_i, f_i\mid\forall i\in\Z\}.$$
Therefore,
$$\tau(X_i (v))=X_i(\tau(v)),\quad{\rm where}\,\,X_i\in\{d_i, h_i, e_i, f_i\mid\forall i\in\Z\}, v\in V(\eta, \epsilon,  \theta).$$
From
$$\phi(C(1\otimes\cdots \otimes1\otimes v))=C(\phi (1\otimes\cdots \otimes1\otimes v)), \,\,\forall\,v\in V(\eta, \epsilon,  \theta),$$  we see that $\tau(C(v))=C(\tau(v))$.
Thus, $\tau$ is a nonzero $\mathfrak L$-module homomorphism. Since $V(\eta, \epsilon, \theta)$ and $V(\eta', \epsilon', \theta')$ are simple $\mathfrak L$-modules, $\tau$ is an  $\mathfrak L$-module isomorphism.  We complete the proof.
\end{proof}
\section{Comparison of tensor product modules with known non-weight modules}
In this section, we compare the tensor product modules constructed in the present paper with all other known non-weight $\mathfrak L$-modules, i.e., $U(\mathfrak{h})$-free modules  of rank one and Whittaker modules (cf. \cite{CH1}). We fix an irreducible tensor module $\mathbf{T}^{M}$ as defined in \eqref{a122}, where $\lambda_k, \alpha_k\in\C^*, \beta_k, \gamma_k \in\C$ for $k=1, 2, \ldots, m$ with the $\lambda_k$ pairwise distinct and  $2\b_k\notin \Z_+$ when $M=\Theta$.

Let $\underline{\mu}=(\mu_1,\ldots,\mu_4)\in\C^4$. Assume that $J_{\underline{\mu}}$ is the left ideal of $U(\mathfrak{L}_{+})$ generated by $\{d_1-\mu_1, d_2-\mu_2, e_0-\mu_3,  f_1-\mu_4, d_j, e_k, f_l, h_m\mid j\geq 3, k\geq 1, l\geq 2, m\geq 1\}.$
Denote $N_{\underline{\mu}}:=U(\mathfrak{L}_{+})/J$. Then ${\rm Ind}(N_{\underline{\mu}}):=U(\mathfrak L)\otimes_{U(\mathfrak{L}_{+})} N_{\underline{\mu}}$ is a universal Whittaker module, and any Whittaker module is a quotient of some universal Whittaker module.

For any $r\in\Z_+, l, j\in\Z$, as in \cite{LLZ}, we denote
$$\omega_{l,j}^{(r)}=\sum_{i=0}^r\binom{r}{i}(-1)^{r-i}d_{l-j-i}d_{j+i}\in U(\mathfrak L).$$
\begin{lemm}\label{theo123456} Keep notations as above. Then the following statements hold.
\begin{itemize}
\item[(1)] For sufficiently large $i\in\N$, the action of $d_i$ on  $\mathbf{T}^{M}$ is not locally finite.
\item[(2)] If $V(\eta, \epsilon, \theta)$ is not the $1$-dimensional trivial module, then for any $r\in\Z_+$, there exist $l, j\in\Z$ and $v\in V(\eta, \epsilon, \theta)$ such that $\omega_{l,j}^{(r)}(1\otimes\cdots\otimes1\otimes v)\neq0$.
\item[(3)] Assume that $m\geq2$ and $V(\eta, \epsilon, \theta)$ is the $1$-dimensional trivial module. Then for any $r>2$, there exist $l, j\in\Z$ such that $\omega_{l,j}^{(r)}(1\otimes\cdots\otimes1)\neq0$.
\end{itemize}
\end{lemm}
\begin{proof}

(1) For any $g(\mathbf{s}, \mathbf{t})\in\mathbf{T}^{M}$ and sufficiently large $i\in\N$, it follows from Proposition \ref{prop2} that ${\rm deg\,}(d_i(g(\mathbf{s}, \mathbf{t})))\succ {\rm deg\,}(g(\mathbf{s}, \mathbf{t}))$, so we have $g(\mathbf{s}, \mathbf{t}), d_i(g(\mathbf{s}, \mathbf{t})), d_i^2(g(\mathbf{s}, \mathbf{t})), \ldots$ are linearly independent. Hence, (1) follows.


(2) Take $v$ to be the highest weight vector of  $V(\eta, \epsilon, \theta)$. It is important to observe  that  the vectors  $v, d_{-2}v, d_{-3}v, \ldots, d_{-r-2}v$ are  linearly independent in $V(\eta, \epsilon, \theta)$, since they belong to different weight subspaces of $V(\eta, \epsilon, \theta)$. Take $l=r+1$ and $j=-r-2$. We compute
\begin{eqnarray*}
\omega_{l,j}^{(r)}(1\otimes\cdots\otimes1\otimes v)&=&\sum_{i=0}^r\binom{r}{i}(-1)^{r-i}d_{l-j-i}d_{j+i}(1\otimes\cdots\otimes1\otimes v)\\
&=&\sum_{i=0}^r\binom{r}{i}(-1)^{r-i}\big(d_{l-j-i}(1\otimes\cdots\otimes1)\otimes d_{j+i}(v)\\
&&+d_{l-j-i}(d_{j+i}(1\otimes\cdots\otimes1))\otimes v\big),
\end{eqnarray*}
which is nonzero.

(3) In this case, we identify  $\mathbf{T}^{M}$ with $\C[\mathbf{s}, \mathbf{t}]$ as a vector space. We compute
\begin{eqnarray*}
\omega_{l,j}^{(r)}(1)&=&\sum_{i=0}^r\binom{r}{i}(-1)^{r-i}d_{l-j-i}d_{j+i}(1)\\
&=&\sum_{i=0}^r\binom{r}{i}(-1)^{r-i}d_{l-j-i}\big(\sum_{k=1}^m\lambda_{k}^{j+i}(s_k+(j+i)\r_k)\big)\\
&=&\sum_{i=0}^r\binom{r}{i}(-1)^{r-i}\sum_{k=1}^m\sum_{k'\in\{1,\ldots,m\}\setminus\{k\}}\lambda_{k'}^{l-j-i}\lambda_{k}^{j+i}(s_{k'}+(l-j-i)\gamma_{k'})(s_k+(j+i)\r_k)\\
&&+\sum_{i=0}^r\binom{r}{i}(-1)^{r-i}\sum_{k=1\atop}^m\lambda_{k}^l\big(s_k+(l-j-i)\r_k\big)\big(s_k+(j+i)\r_k-l+j+i\big).
\end{eqnarray*}
The following identity
$$\sum_{i=0}^r\binom{r}{i}(-1)^{r-i}i^j=0,\,\,\forall j, \,r\in\Z_+\,\,{\rm with}\,\,j<r$$
forces that the second summand of $\omega_{l,j}^{(r)}(1)$ vanishes provided $r>2$. Since $r>2$, the coefficient of $s_1s_2$ in $\omega_{l,j}^{(r)}(1)$ is
$$\sum_{i=0}^r\binom{r}{i}(-1)^{r-i}(\lambda_1^{l-j-i}\lambda_2^{j+i}+\lambda_2^{l-j-i}\lambda_1^{j+i})=\big(\lambda_1^{l-j-r}\lambda_2^{j}+(-1)^r\lambda_2^{l-j-r}\lambda_1^{j}\big)
(\lambda_2-\lambda_1)^r,$$
which is nonzero provided that $(\lambda_1/\lambda_2)^{l-2j-r}+(-1)^r\neq0$. Thus (3) follows.
\end{proof}

Finally, we have the following result which asserts that the tensor product modules constructed in the present paper are different from the other known non-weight modules.
\begin{prop}\label{theovvn}
The tensor product module $\mathbf{T}^{M}$ is a new non-weight $\mathfrak L$-module.
\end{prop}
\begin{proof}
For the case $m=1$, the assertion has been proved  in \cite[Proposition 5.2]{CY}. Next we assume that $m\geq2$ in the following.
We need to show that the tensor product module $\mathbf{T}^{M}$ is neither isomorphic to a Whittaker module nor isomorphic to a  $U(\mathfrak{h})$-free modules of rank one in \cite{CH1}. For that, let $W$ be a Whittaker module, then $W$ is isomorphic to a quotient of $N_{\underline{\mu}}$ for some $\underline{\mu}=(\mu_1,\ldots,\mu_4)\in\C^4$. Noticing that the action of $d_i$ on $W$ is locally finite for sufficiently large $i\in\N$.  It follows from Lemma \ref{theo123456} (1) that $\mathbf{T}^{M}\ncong W$.  Combining
\cite[Lemma 5.1(2)]{CY} with Lemma \ref{theo123456} (2)  and (3), we see that $\mathbf{T}^{M}$ is not isomorphic to the irreducible non-weight modules defined in \cite{CH1}. We complete the proof.
\end{proof}


\end{document}